\documentclass[reqno,english]{amsart}
\usepackage{amsfonts,amsmath,latexsym,verbatim,amscd,mathrsfs,color,array}
\usepackage[colorlinks=true]{hyperref}
\usepackage{amsmath,amssymb,amsthm,amsfonts,graphicx,color}
\usepackage{amssymb}
\usepackage{pdfsync}
\usepackage{epstopdf}
\usepackage{cite}
\usepackage{graphicx}

\newcommand{\bt}{\begin{theorem}}
\newcommand{\et}{\end{theorem}}
\newcommand{\bl}{\begin{lemma}}
\newcommand{\el}{\end{lemma}}
\newcommand{\bd}{\begin{definition}}
\newcommand{\ed}{\end{definition}}
\newcommand{\bc}{\begin{corollary}}
\newcommand{\ec}{\end{corollary}}
\newcommand{\bp}{\begin{proof}}
\newcommand{\ep}{\end{proof}}
\newcommand{\bx}{\begin{example}}
\newcommand{\ex}{\end{example}}
\newcommand{\bi}{\begin{exercise}}
\newcommand{\ei}{\end{exercise}}
\newcommand{\bo}{\begin{prop}}
\newcommand{\eo}{\end{prop}}
\newcommand{\br}{\begin{remark}}
\newcommand{\er}{\end{remark}}
\newcommand{\be}{\begin{equation}}
\newcommand{\ee}{\end{equation}}
\newcommand{\ba}{\begin{align}}
\newcommand{\ea}{\end{align}}
\newcommand{\bn}{\begin{enumerate}}
\newcommand{\en}{\end{enumerate}}
\newcommand{\bg}{\begin{align*}}
\newcommand{\bcs}{\begin{cases}}
\newcommand{\ecs}{\end{cases}}

\newcommand{\bean}{\begin{eqnarray*}}
\newcommand{\eean}{\end{eqnarray*}}

\newtheorem{example}{Example}[section]
\newtheorem{definition}{Definition}[section]
\newtheorem{theorem}{Theorem}[section]
\newtheorem{lemma}{Lemma}[section]
\newtheorem{cor}{Corollary}[section]
\newtheorem{prop}{Proposition}[section]
\newtheorem{remark}{Remark}[section]

\numberwithin{equation}{section}

\begin{document}
\title[Nondegeneracy of harmonic maps]{Nondegeneracy of harmonic maps from $\mathbb R^2$ to $\mathbb S^2$}
\author[G. Chen]{Guoyuan Chen}
\address{\noindent
School of Data Sciences, Zhejiang University of Finance \& Economics, Hangzhou 310018, Zhejiang, P. R. China}
\email{gychen@zufe.edu.cn}

\author[Y. Liu]{Yong Liu}
\address{\noindent School of Mathematics and Physics, North China Electric Power University, Beijing, China}
\email{liuyong@ncepu.edu.cn}

\author[J. Wei]{Juncheng Wei}
\address{\noindent
Department of Mathematics,
University of British Columbia, Vancouver, B.C., Canada, V6T 1Z2}
\email{jcwei@math.ubc.ca}

\begin{abstract}
We prove that all harmonic maps from $\mathbb R^2$ to $\mathbb S^2$ with finite energy are nondegenerate. That is, for any harmonic map $u$ from $\mathbb R^2$ to $\mathbb S^2$ of degree $m$ (in $\mathbb Z$), all bounded kernel maps of the linearized operator $L_u$ at $u$ are generated by these harmonic maps near $u$ and hence the real dimension of bounded kernel space of $L_u$ is $4|m|+2$.
\end{abstract}
\maketitle
\section{Introduction}

In this paper, we consider the harmonic maps given by
\begin{eqnarray}\label{e:harmonic-map}
\Delta u + |\nabla u|^2u=0, \ \ u:\mathbb R^2\to \mathbb S^2,
\end{eqnarray}
where $\mathbb S^2=\{u=(u_1,u_2,u_3):\,|u|=1\}$ is the unit sphere in $\mathbb R^3$, $\Delta $ is the Laplacian in $\mathbb R^2$ and
$$|\nabla u|^2=\sum_{j=1}^2 \sum_{i=1}^3 \left(\frac{\partial u_i}{\partial x_j}\right)^2.$$
Harmonic maps are critical points of the following associated energy
\begin{eqnarray}\label{e:energy1}
\mathcal E(u)=\int_{\mathbb R^2}|\nabla u|^2dx.
\end{eqnarray}
Harmonic maps also can be defined between general Riemannian manifolds. We refer the interested readers to  \cite{eells1978report}, \cite{eells1988another}, \cite{schoen1997lectures}, \cite{lin2008analysis}, \cite{helein2008harmonic} and the references therein for more results in this direction.

One important case is the harmonic maps between Riemann surfaces.
It is well known that oriented two-dimensional Riemannian manifold has a natural complex structure. Hence any holomorphic or anti-holomorphic map between oriented surfaces is a harmonic map. Conversely, harmonic map is not necessarily holomorphic or anti-holomorphic. There is a topological obstruction. For example, if a harmonic map $u$ from a surface $M$ to $\mathbb S^2$ satisfies certain degree conditions (for example, $|{\rm deg}(u)|$ is larger than the genus of $M$), then $u$ is holomorphic or anti-holomorphic. For more related results, see for example \cite{eells1964harmonic}, \cite{eells1978report}, \cite{eells1988another}, \cite{jost2006harmonic}, \cite{helein2008harmonic}, and the references therein.

A map $u$ from $\mathbb R^2$ to $\mathbb S^2$ is harmonic if and only if $u$ is holomorphic or anti-holomorphic. In particular, if we choose local coordinates of $\mathbb S^2$ as stereographic projections $\mathcal S$ and $\mathcal S'$ (see (\ref{e:stereo1}), (\ref{e:stereo2}) below) and consider $\mathbb R^2$ as the complex plane, then a map $u$ from $\mathbb R^2$ to $\mathbb S^2$ with finite topological degree is harmonic if and only if there is a (irreducible) rational complex valued function $f=q/p$ in $\mathbb C$ such that
$u=\mathcal S(f)$. Here $q$ and $p$ are complex polynomials with algebraic order $l$ and $n$ respectively. Some basic computations imply that ${\rm deg}(u)=\max\{l,n\}$ (see Section \ref{e:classification} below). We should mention that the classifications of harmonic maps from $\mathbb C$ to Lie groups or other symmetric spaces were also investigated, see, for example, \cite{uhlenbeck1989harmonic}, \cite{eells1983harmonic}, \cite{guest1997harmonic}, and the references therein.

It is clear that changing the coefficients (complex numbers) of $f$ continuously yields a family of harmonic maps.
Therefore, it generates kernel maps for the linearized operators $L_u$
$$ L_u [v]:= \Delta v +|\nabla u |^2 v + 2 (\nabla u \cdot \nabla v) u $$
of (\ref{e:harmonic-map}) at some fixed harmonic maps $u$. Let us consider a simple example. Assume $u(z)=\mathcal S\left(\frac{z^2+1}{z}\right)$. The degree of $u$ is $2$. Let $$u(z;a_0,a_1,a_2;b_0,b_1,b_2)=\mathcal S\left(\frac{(1+a_0)z^2+a_1z+1+a_2}{b_0z^2+(1+b_1)z+b_2}\right)$$
with some small complex numbers $a_0,a_1,a_2;b_0,b_1,b_2$. Then $u(z;a_0,a_1,a_2;b_0,b_1,b_2)$ are close to $u(z)$. Then this family of harmonic maps generates $10$ linearly independent bounded kernels of $L_u$.

Our main theorem asserts the nondegeneracy of all harmonic maps. We prove that
\begin{theorem}\label{t:general}
Let $u$ be a harmonic map from $\mathbb R^2$ to $\mathbb S^2$ of degree $m\in \mathbb Z$. Then all the bounded maps in the kernel of $L_u$ are generated by harmonic maps close to $u$. In particular, the real dimension of the bounded kernel space of $L_u$ is $4|m|+2$.
\end{theorem}

\begin{remark}
(1) We should point out that the degree $m$ of harmonic maps can be negative in Theorem \ref{t:general}. It corresponds to anti-holomorphic functions. See Section \ref{e:classification} below. Moreover, we emphasize that the harmonic maps we discussed here are general, especially, we do not restrict the harmonic maps to be \emph{corotational}. For nondegeneracy within corotational classes, see \cite{GKT1} and \cite{GKT2}.

(2) All bounded kernels of $L_u$ can be written down explicitly. See the proof of Proposition \ref{p:cr}. As an example, we will give the explicit formula of kernels for $m$-corotational case (Corollary \ref{t:main}).
\end{remark}

\begin{remark}
(1) For $m=1$, the nondegeneracy of $1$-corotational harmonic maps was obtained by \cite{davila2017singularity}. For case of half-harmonic map from $\mathbb R$ to $\mathbb S^1$ of degree $1$, the nondegeneracy was verified by \cite{sire2017nondegeneracy}. In \cite{chanillo2005asymptotic}, the authors proved the nondegeneracy of standard bubbles of degree $1$ for $\Delta u=2u_x\times u_y$ where $u$ is a map from $\mathbb R^2$ to $\mathbb R^3$.

(3) Since the set of harmonic maps from $\mathbb R^2$ to $\mathbb S^2$ with finite topological degree  corresponds to the space of complex rational functions by stereographic projection, it introduces a natural algebraic structure. Similar phenomena were found in Toda system, see \cite{wei2011nondegeneracy} \cite{lin2012classification}.

(2) Theorem \ref{t:general} asserts that the kernel space of $L_u$ is generated by changing the $2m+2$ complex coefficients of corresponding rational functions. Such kind of result was obtained for half wave maps from $\mathbb R$ to $\mathbb S^2$ in a recent paper \cite{lenzmann2017energy}. To be more precise, traveling solitary waves from $\mathbb R$ to $\mathbb S^2$ with degree $m$ can be generated by finite Blaschke products with degree $m$ which are holomorphic or anti-holomorphic functions. These finite Blaschke products depend on $2m+1$ real parameters. The authors proved that the null space of linearized operators is generated by differentiation these parameters and $x$-, $y$-rotations. An important ingredient in the proof is the classification and nondegeneracy of half-harmonic maps from $\mathbb R$ to $\mathbb S^1$. We should point out that the proof of  nondegeneracy in \cite{lenzmann2017energy} is done through analyzing the spectrum of linearized operator.

Our proof relies on the fact that all harmonic maps from $\mathbb R^2$ to $\mathbb S^2$ are minimizers (locally stable critical points) in its degree class (see (\ref{e:energy-first-order}) and (\ref{e:minus-degree}) below). That is, $u$ is a harmonic map with degree $m\ge 0$ (resp. $m<0$) if and only if $u$ satisfies the first order equation (\ref{e:JH}) (resp. (\ref{e:anti-JH})). Let $L_{1,u}$ be the linearized (first order) operator corresponding to (\ref{e:energy-first-order}) at $u$. Then we prove that for $m\ge 0$ the bounded kernel space of $L_u$ is the same as that of $L_{1,u}$ (see Proposition \ref{l:Lu-L1u} below). Then we verify that all
maps in the kernel of $L_{1,u}$ satisfies Cauchy-Riemann equations after stereographic projection (see Proposition \ref{p:cr} below). For $m<0$, the similar results correspond to anti-holomorphic and anti-version of Cauchy-Riemann equations.
\end{remark}

As a special case, we consider $m$-corotational harmonic maps. That is, $U_m(z)=\mathcal S(z^m)$, $m\in \mathbb N$, in $\mathbb C$.
In polar coordinates $(r,\theta)$ in $\mathbb R^2$,
the $m$-corotational harmonic maps can be written as
\begin{eqnarray}\label{e:standard}
U_m(r,\theta)=\left(
              \begin{array}{c}
                \cos m\theta\sin Q_m(r) \\
                \sin m\theta\sin Q_m(r) \\
                \cos Q_m(r) \\
              \end{array}
            \right)=\left(
                      \begin{array}{c}
                        \dfrac{2r^m\cos m\theta}{1+r^{2m}} \\
                        \dfrac{2r^m\sin m\theta}{1+r^{2m}} \\
                        \dfrac{r^{2m}-1}{r^{2m}+1} \\
                      \end{array}
                    \right),
\end{eqnarray}
where
$Q_m(r)=\pi-2\arctan (r^m).$
Theorem \ref{t:general} implies the explicit bounded kernel of the linearized operator at the standard $m$-corotational harmonic map $U_m$.
For simplicity of notations, we set
\begin{equation*}
E_1=\left(
             \begin{array}{c}
               -\sin m\theta \\
               \cos m\theta  \\
               0\\
             \end{array}
           \right)=\left(
                     \begin{array}{c}
                       ie^{im\theta} \\
                       0 \\
                     \end{array}
                   \right)
           ,\quad E_2=\left(
                                       \begin{array}{c}
                                         \cos m\theta\cos Q_m \\
                                         \sin m\theta\cos Q_m \\
                                         -\sin Q_m \\
                                       \end{array}
                                     \right)=\left(
                                               \begin{array}{c}
                                                 e^{im\theta}\cos Q_m \\
                                                 -\sin Q_m \\
                                               \end{array}
                                             \right)
                                     .
\end{equation*}
Then we have the following corollary.
\begin{cor}\label{t:main}
Assume that $m\ge 1$. The standard $m$-corotational harmonic map $U_m$ is nondegenerate. That is, all bounded maps in the kernel of $L_{U_m}$ are linear combinations of $\{E_{kj}, \tilde E_{\nu j}, \tilde E_{mj}\}$ ($k=0,1,\cdots,m$, $\nu=1,\cdots,m-1$ and $j=1,2$). Here
\begin{eqnarray}\label{e:ek1}
E_{k1}=\frac{r^{m-k}}{1+r^{2m}}(-\cos k\theta E_1 +\sin k\theta E_2),\quad
 E_{k2}=\frac{r^{m-k}}{1+r^{2m}}(\sin k\theta E_1 +\cos k\theta E_2),
\end{eqnarray}
where $k=0,1,\cdots, m$, and
\begin{eqnarray}\label{e:ek3}
\tilde E_{\nu1}=\frac{r^{m+\nu}}{1+r^{2m}}(-\cos \nu\theta E_1 +\sin \nu\theta E_2),
\quad
\tilde  E_{\nu2}=\frac{r^{m+\nu}}{1+r^{2m}}(\sin \nu\theta E_1 +\cos \nu\theta E_2),
\end{eqnarray}
where $\nu=1,\cdots, m-1$, and
\begin{eqnarray}\label{e:ek5}
\tilde E_{m1}=\left(
                \begin{array}{c}
                  0 \\
                  \dfrac{r^{2m}-1}{r^{2m}+1} \\
                  -\dfrac{2r^m\sin m\theta}{r^{2m}+1} \\
                \end{array}
              \right)
,
\quad
\tilde  E_{m2}=\left(
                \begin{array}{c}
                  \dfrac{r^{2m}-1}{r^{2m}+1} \\
                  0\\
                  -\dfrac{2r^m\cos m\theta}{r^{2m}+1} \\
                \end{array}
              \right),
\end{eqnarray}
In particular, $\dim {\rm ker}L_{U_m}=4m+2$.
\end{cor}

\begin{remark}
(1) For $m\le -1$, $U_m(z)=\mathcal S(\bar z^m)$. The similar results can be verified by reversing the rotation direction.

(2) Nondegeneracy for $m$-corotational harmonic map $U_m$ can also be proved by ODE method. See Appendix below.
\end{remark}

This paper is organized as follows. In Section 2, we recall some basic classifications results and compute the degree of general harmonic maps in terms of algebraic order of rational functions. Section 3 is devoted to the proof of Theorem \ref{t:general}. In Section 4, we give an explicit formula of the kernel maps in $m$-corotational case. In Appendix, we illustrate some details of computations, and prove Corollary \ref{t:main} by ODE methods.

\subsection*{Acknowledgement}
G. Chen is partially supported by Zhejiang Provincial Science Foundation of China (No. LY18A010023) and Zhejiang University of Finance and Economics.
J. Wei is partially supported by NSERC of Canada. Part of this work was finished while the first two authors were visiting the University of British Columbia in 2017.

\section{Preliminaries}\label{e:classification}

In this section, we first recall the classification facts for harmonic maps from $\mathbb R^2$ to $\mathbb S^2$ as well as compute their degrees.

It is well known that $\mathbb S^2$ is a complex manifold. For our application, we give a system of local charts by the stereographic projection. Let $z=(x,y)\in \mathbb R^2=\mathbb C$ and $s=(s_1,s_2,s_3)\in \mathbb S^2$. Define
\begin{eqnarray}\label{e:stereo1}
\mathcal S:\mathbb C\to \mathbb S^2\setminus\{N\}  \quad \mbox{by } s_1=\frac{2x}{1+|z|^2},\,s_2=\frac{2y}{1+|z|^2}\mbox{ and }s_{3}=\frac{|z|^2-1}{|z|^2+ 1}.
\end{eqnarray}
and
\begin{eqnarray}\label{e:stereo2}
\mathcal S':\mathbb C\to \mathbb S^2\setminus\{S\}  \quad \mbox{by } s_1=\frac{2x}{1+|z|^2},\,s_2=-\frac{2y}{1+|z|^2}\mbox{ and }s_{3}=\frac{1-|z|^2}{1+|z|^2},
\end{eqnarray}
Here $S=(0,0,-1)$ and $N=(0,0,1)$ denote the south and north pole respectively. Alternatively, in complex variable form,
\begin{eqnarray}\label{e:steregraphic-projection}
\mathcal S=\frac{1}{1+|z|^2}\left(
               \begin{array}{c}
                 2 z \\
                 |z|^2-1 \\
               \end{array}
             \right),\quad\mathcal S'=\frac{1}{1+|z|^2}\left(
               \begin{array}{c}
                 2\bar z \\
                 1-|z|^2 \\
               \end{array}
             \right).
\end{eqnarray}
Therefore, the transition function between these two charts is
$\frac{1}{z}$. It is holomorphic in $\mathbb C\setminus \{0\}$.

Topological degree of a $C^1$ map $u$ from $\mathbb R^2$ to $\mathbb S^2$ can be defined by the de Rham approach
\begin{eqnarray}\label{e:de-rham-degree}
{\rm deg}(u)=\frac{1}{4\pi}\int_{\mathbb R^2}u\cdot(u_y\times u_{x})=\frac{1}{4\pi}\int_{\mathbb R^2} u_{x}\cdot (u\times u_y).
\end{eqnarray}
It is well known that (\ref{e:de-rham-degree}) is equivalent to the Brouwer's degree for all $C^1$ maps. See, for example, \cite[Chapter III]{outerelo2009mapping}.

Since $u\perp u_{x}$ and $u\perp u_{y}$,
it holds that
$|u\times u_{y}|=|u||u_{y}|$. Then
we have
\begin{eqnarray}
\mathcal E(u)=\frac{1}{2}\int_{\mathbb R^2}(|u_{x}|^2+|u\times u_{y}|^2)dxdy.\notag
\end{eqnarray}
Rewrite
\begin{eqnarray}\label{e:energy-first-order}
\mathcal E(u)&=&\frac{1}{2}\int_{\mathbb R^2}|u_{x}-u\times u_{y}|^2+\int_{\mathbb R^2} u_{x}\cdot (u\times u_y)\notag\\
&=&\frac{1}{2}\int_{\mathbb R^2}|u_{x}-u\times u_{y}|^2+4\pi {\rm deg}(u).
\end{eqnarray}
If
\begin{eqnarray}\label{e:JH}
u_{x}=u\times u_{y},
\end{eqnarray}
then $\mathcal E(u)=4\pi {\rm deg}(u)$. Hence ${\rm deg}(u)\ge 0$. Since Brouwer's degree is invariant by homotopy deformations, we find that $u$ is a minimizer in its degree class.

On the other hand, we also have that
\begin{eqnarray}\label{e:minus-degree}
\mathcal E(u)&=&\frac{1}{2}\int_{\mathbb R^2}(|u\times u_{x}|^2+|u_{y}|^2)dxdy\notag\\
&=&\frac{1}{2}\int_{\mathbb R^2}|u\times u_{x}-u_{y}|^2dxdy+\int_{\mathbb R^2}(u\times u_{x})\cdot u_{y}dxdy\notag\\
&=&\frac{1}{2}\int_{\mathbb R^2}|u\times u_{x}-u_{y}|^2dxdy-4\pi {\rm deg}(u).
\end{eqnarray}
Hence if $u$ satisfies
\begin{eqnarray}\label{e:anti-JH}
u_{y}=u\times u_{x},
\end{eqnarray}
then $\mathcal E(u)=-4\pi  {\rm deg}(u)$. It follows that ${\rm deg}(u)\le0$. Similarly by the invariance of deformations of Brouwer's degree, we have that $u$ is a local minimizer in its degree class.

Before discussing the nondegeneracy problem, we first recall the classification result.
\begin{prop}\label{p:harmonic-holomorphic}
A map $u$ from $\mathbb R^2$ to $\mathbb S^2$ is harmonic if and only if $u$ is holomorphic or anti-holomorphic.
\end{prop}

\begin{remark}
For the relation between harmonic maps and holomorphic (or anti-holomorphic) maps from general surfaces to $\mathbb S^2$, we refer the reader to \cite[11.6]{eells1978report}, \cite[Section 2.2]{helein2008harmonic} and the references therein.
\end{remark}

Note that if $u$ satisfies (\ref{e:JH}) or (\ref{e:anti-JH}), then $u$ is a critical point of $\mathcal E$. That means that $u$ is a solution of (\ref{e:harmonic-map}). Conversely, Proposition \ref{p:harmonic-holomorphic}, Lemma \ref{l:first-cr} and Remark \ref{r:anti-cr} below tell us that if $u$ is a solution of (\ref{e:harmonic-map}), then $u$ satisfies (\ref{e:minus-degree}) or (\ref{e:anti-JH}). Hence we find that all harmonic maps from $\mathbb R^2$ to $\mathbb S^2$ are stable.

By Proposition \ref{p:harmonic-holomorphic}, in local coordinates, the harmonic maps are holomorphic or anti-holomorphic functions.
\begin{cor}
A map $u$ from $\mathbb R^2$ to $\mathbb S^2$ is harmonic with ${\rm deg}(u)\ge 0$ (resp. ${\rm deg}(u)< 0$) if and only if $u=\mathcal S(p/q)$ where $p$ and $q$ are complex polynomials of $z$ (resp. $\bar z$). That is, $u$ is a harmonic map if and only if $u$ is a stereographic projection of a rational function of $z$ (resp. $\bar z$). We call (irreducible) rational function $p/q$ the generate function function of harmonic map $u$. Moreover, all harmonic maps from $\mathbb R^2$ to $\mathbb S^2$ are stable.
\end{cor}
\begin{remark}
In what follows, we shall focus on the nonnegative degree harmonic maps. Harmonic maps of negative degree correspond to anti-holomorphic functions. The proof of nondegeneracy is similar.
\end{remark}

We now in the position to compute the degree of harmonic maps. The degree of rational functions was computed in \cite{segal1979topology} (see also \cite{cohen1991topology}). For our utilities, we give a basic computation. Let us investigate some examples first.
\begin{example}
Some direct computations verify that
\begin{itemize}
  \item[(1)] ${\rm deg}(\mathcal S( z))={\rm deg}(\mathcal S(\frac{1}{z}))=1$.
${\rm deg}(\mathcal S(z^2))={\rm deg}(\mathcal S(\frac{1}{z^2}))=2$.
${\rm deg}(\mathcal S\circ z^m)={\rm deg}(\mathcal S(\frac{1}{z^m}))=m$ with $m\ge 0$.
  \item[(2)] ${\rm deg}(\mathcal S(z+\frac{1}{z}))=2$.
  \item[(3)] ${\rm deg}(\mathcal S(\bar z))={\rm deg}(\mathcal S(\frac{1}{\bar z}))=-1$. And ${\rm deg}(\mathcal S(\bar z^m))={\rm deg}(\mathcal S(\frac{1}{\bar z^m}))=-m$ with $m>0$.
\end{itemize}
\end{example}

In general, we have the following result.
\begin{prop}\label{p:degree}
Assume that $u=\mathcal S (g+\dfrac{p}{q})$ where $g$, $p(\ne 0)$ and $q$ are three polynomial complex functions with order $s$, $l$ and $t$ respectively, $l<t$, and $p/q$ is irreducible. Then
\begin{eqnarray}\label{e:degree-f-g}
{\rm deg}(u)=s+t.
\end{eqnarray}
\end{prop}
\begin{proof}
Let $f(z)=g(z)+\dfrac{p(z)}{q(z)}$. Then $f(z)=0$ yields that
\begin{eqnarray}\label{e:polynomial}
g(z)q(z)+p(z)=0.
\end{eqnarray}
Without loss of generality, we may assume that there are $s+t$ different roots, $a_1,\cdots, a_s,a_{s+1},\cdots, a_{s+t}$, for (\ref{e:polynomial}).

Assume that $q(z)=0$ has roots $b_1,\cdots, b_t$. It is clear that $b_1,\cdots, b_t$ can not be roots of (\ref{e:polynomial}). Now we prove that in $\Omega:=\mathbb C\setminus \{b_1,\cdots, b_t\}$, ${\rm deg}(f)=s+t$. Indeed, since $\mathcal S\circ f$ maps $b_1,\cdots, b_t$ to $N$,
\begin{eqnarray}
{\rm deg}(f,\Omega, 0)=\sum_{j=1}^{s+t}{\rm deg}(f,B_{\varepsilon}(a_j), 0),\notag
\end{eqnarray}
where $\varepsilon>0$ is a sufficient small constant. Recall that
\begin{eqnarray}
{\rm deg}(f,B_{\varepsilon}(a_j), 0)=1,\quad \forall \varepsilon>0 \mbox{ sufficiently small},\notag
\end{eqnarray}
and
\begin{eqnarray}
{\rm deg}(\mathcal S,B_{\delta}(a_j), S)=1.\notag
\end{eqnarray}
Hence by formula for degree of composition maps, we obtain (\ref{e:degree-f-g}). This completes the proof.
\end{proof}

\begin{remark}
(1) In general, a rational function $f(z)$ in $\mathbb C$ can be represented by
\begin{eqnarray}
f(z)=\sum_{l=0}^s a_{s-l}z^l+\sum_{k=1}^K\sum_{j=1}^{J_k}\frac{a_{k,j}}{(z-z_k)^j},\quad \mbox{for }z\in \mathbb C\setminus\{z_1,\cdots,z_K\}.\notag
\end{eqnarray}
where $a_s\ne0$ and $a_{k,J_k}\ne 0$ ($k=1,\cdots,K$). In what follows, we call $J_k$ the singular order of $z_k$.
Hence the degree of harmonic map $u=\mathcal S(f)$ is
\begin{eqnarray}
{\rm deg}(u)=s+\sum_{k=1}^K J_k.\notag
\end{eqnarray}

(2) If we write (irreducible) rational functions as $f=p/q$ where $p$ and $q$ are two polynomial complex functions with order $l$ and $n$, then
\begin{eqnarray}
{\rm deg} ({\mathcal S } (f))=\max \{l,n\}. \notag
\end{eqnarray}

\end{remark}

We now consider the relation between harmonic maps $u(z)=\mathcal S(f)(z)$ and $w(z)=\mathcal S(1/f)(z)$ where $f(z)=f_1(z)+if_2(z)$ is a rational function. That is,
\[
u_{1}=\frac{2f_1}{|f|^2+1},\quad
u_{2}=\frac{2f_2}{|f|^2+1},\quad
u_{3}=\frac{|f|^2-1}{|f|^2+1}.
\]
Let $\tilde u=Q_{1,\alpha}(u)$ be a rotation of $u$ with respect to $u_1$-axis. Hence,
\[
\tilde u_{1}=\frac{2f_1}{|f|^2+1}
\]
\[
\tilde u_{2}=\frac{2}{|f|^2+1}(f_2\cos \alpha-(|f|^2-1)\sin
\alpha)
\]%
\[
\tilde u_{3}=\frac{2}{|f|^2+1}(f_2\sin \alpha+(|f|^2-1)\cos
\alpha)\]%

Some direct computations yield that
\begin{eqnarray}\label{e:u1-rotation}
\mathcal S^{-1}(Q_{1,\alpha}(u))=\frac{u_1+i\tilde u_2}{1-\tilde u_3}=-i\cot\frac{\alpha}{2}+\frac{1}{\sin^2\frac{\alpha}{2}(f-i\cot\frac{\alpha}{2})}.
\end{eqnarray}
In particular, let $\alpha=\pi$, then (\ref{e:u1-rotation}) implies
\begin{eqnarray}
\frac{u_1+i\tilde u_2}{1-\tilde u_3}=-i\cot\frac{\alpha}{2}+\frac{1}{\sin^2\frac{\alpha}{2}(f-i\cot\frac{\alpha}{2})}=\frac{1}{f}.\notag
\end{eqnarray}
That means that the difference between $u(z)=\mathcal S(f)$ and $w(z)=\mathcal S(\frac{1}{f})$ is a rotation.

Moreover, note that
the tangent map of $Q_{1,\alpha}$ is given by
\begin{eqnarray}
v:=\frac{d}{d\alpha}Q_{1,\alpha}(u)|_{\alpha=0}=\left(
                                            \begin{array}{c}
                                              0 \\
                                              -u_3 \\
                                              u_2 \\
                                            \end{array}
                                          \right).\notag
\end{eqnarray}

\section{Proof of Theorem \ref{t:general}}\label{s:proof}
In this section, we shall prove the nondegeneracy of harmonic maps.

Let us compute the associated linearized operators at harmonic maps first.
The associated linearized operators of (\ref{e:harmonic-map}) and (\ref{e:JH}) at $u$ are given by
\begin{eqnarray}
L_u[v]=\Delta v+|\nabla u|^2v+2(\nabla u\cdot\nabla v)u,\notag
\end{eqnarray}
and,
\begin{eqnarray}
L_{1,u}[v]= v_x-v\times u_y-u\times v_y,\notag
\end{eqnarray}
respectively, where $v$ is a smooth map from $\mathbb C$ to $T\mathbb S^2$. Let $C^2_b(\mathbb C,T\mathbb S^2)$ be the space of maps $v$ from $\mathbb C$ to $T\mathbb S^2$ satisfying $\sup_{z\in \mathbb C}(|v(z)|+|\partial v(z)|+|\partial^2 v(z)|)$. Define
\begin{eqnarray}
{\rm ker} L_u:=\{v\in C^2_b(\mathbb C,T\mathbb S^2)\,|\, L_u[v]=0\}\,\, \mbox{and}\,\,{\rm ker} L_{1,u}:=\{v\in C^2_b(\mathbb C,T\mathbb S^2)\,|\, L_{1,u}[v]=0\}.\notag
\end{eqnarray}

\begin{prop}\label{l:Lu-L1u}
Assume that $u$ is a harmonic map with degree $m\ge 0$. It holds that
\begin{eqnarray}\label{e:kernel-1-2}
{\rm ker}L_u={\rm ker}L_{1,u}.
\end{eqnarray}
\end{prop}
\begin{proof}
When $m=0$, (\ref{e:kernel-1-2}) holds. Hence we only consider the case $m>0$.

Since ${\rm deg}(u)>0$ and $\mathcal S^{-1}(u)$ is holomorphic, we find that $|u_x|=|u_y|\ne 0$ almost everywhere in $\mathbb C$.

1. ``${\rm ker}L_{1,u}\Rightarrow {\rm ker}L_{u}$''. From
$v_x-v\times u_y-u\times v_y=0$,
we have that
\begin{eqnarray}
v_y+v\times u_x+u\times v_x=0,\notag
\end{eqnarray}
\begin{eqnarray}
v_{xx}-v_x\times u_y-v\times u_{xy}-u_x\times v_y-u\times v_{xy}=0,\notag
\end{eqnarray}
and
\begin{eqnarray}
v_{yy}+v_y\times u_x+v\times u_{xy}+u_y\times v_x+u\times v_{xy}=0.\notag
\end{eqnarray}
Hence we obtain that
\begin{eqnarray}\label{e:vxxyy}
v_{xx}+v_{yy}-v_x\times u_y+v_y\times u_x-u_x\times v_y+u_y\times v_x=0.
\end{eqnarray}
Since
\begin{eqnarray}
&-(v_x\times u_y)\cdot u=u_x\cdot v_x,\quad (v_y\times u_x)\cdot u=u_y\cdot v_y,\notag\\
&-(u_x\times v_y)\cdot u=u_y\cdot v_y,\quad (u_y\times v_x)\cdot u=u_x\cdot v_x,\notag
\end{eqnarray}
it holds that
\begin{eqnarray}
(v_{xx}+v_{yy})\cdot u+2\nabla u\cdot \nabla v=0.\notag
\end{eqnarray}
Similarly, \eqref{e:vxxyy}$\cdot u_x$ and \eqref{e:vxxyy}$\cdot u_y$  yield that
\begin{eqnarray}
(v_{xx}+v_{yy})\cdot u_x+|\nabla u|^2v\cdot u_x=0,\quad (v_{xx}+v_{yy})\cdot u_y+|\nabla u|^2v\cdot u_y=0.\notag
\end{eqnarray}
Therefore, we obtain that
\begin{eqnarray}
\Delta v+|\nabla u|^2v+2(\nabla u\cdot\nabla v)u=0.\notag
\end{eqnarray}

2. ``${\rm ker}L_{u}\Rightarrow {\rm ker}L_{1,u}$''. We shall prove this by Liouville theorem.
Note that
\begin{eqnarray}\label{e:reduce}
&&(v_x-v\times u_y-u\times v_y)\cdot u=0,\notag\\
&&(v_x-v\times u_y-u\times v_y)\cdot u_x=u_x\cdot v_x-u_y\cdot v_y,\\
&&(v_x-v\times u_y-u\times v_y)\cdot u_y=u_x\cdot v_y+u_y\cdot v_x.\notag
\end{eqnarray}

\emph{Claim: It holds that
\begin{eqnarray}\label{e:uvy}
u_x\cdot v_x-u_y\cdot v_y=0,
\end{eqnarray}
\begin{eqnarray}\label{e:uv-x}
u_x\cdot v_y+u_y\cdot v_x=0.
\end{eqnarray}}

Indeed, we compute $\Delta (u_x\cdot v_x-u_y\cdot v_y)$. By some direct computations, we get that
\begin{eqnarray}
\partial_x(u_x\cdot v_x-u_y\cdot v_y)=u_{xx}\cdot v_x+u_x\cdot v_{xx}-u_{xy}\cdot v_y-u_y\cdot v_{xy},\notag
\end{eqnarray}
\begin{eqnarray}
\partial_x^2(u_x\cdot v_x-u_y\cdot v_y)&=&u_{xxx}\cdot v_x+2u_{xx}\cdot v_{xx}+u_x\cdot v_{xxx}\notag\\
&&-u_{xxy}\cdot v_y-2u_{xy}\cdot v_{xy}-u_{y}\cdot v_{xxy}.\notag
\end{eqnarray}
Similarly, it holds that
\begin{eqnarray}
\partial_y(u_x\cdot v_x-u_y\cdot v_y)=u_{xy}\cdot v_x+u_x\cdot v_{xy}-u_{yy}\cdot v_y-u_y\cdot v_{yy},\notag
\end{eqnarray}
\begin{eqnarray}
\partial_y^2(u_x\cdot v_x-u_y\cdot v_y)&=&u_{xyy}\cdot v_x+2u_{xy}\cdot v_{xy}+u_x\cdot v_{xyy}\notag\\
&&-u_{yyy}\cdot v_y-2u_{yy}\cdot v_{yy}-u_{y}\cdot v_{yyy}.\notag
\end{eqnarray}
Hence we have that
\begin{eqnarray}
\Delta (u_x\cdot v_x-u_y\cdot v_y)&=&(u_{xxx}+u_{xyy})\cdot v_x+2(u_{xx}\cdot v_{xx}-u_{yy}\cdot v_{yy})\notag\\
&&+u_x\cdot (v_{xxx}+v_{xyy})-(u_{xxy}+u_{yyy})\cdot v_y-u_{y}\cdot (v_{xxy}+v_{yyy})\notag\\
&=&\partial_x(\Delta u)\cdot v_x+2((\Delta u)\cdot v_{xx}-u_{yy}\cdot (\Delta v))\notag\\
&&+u_x\cdot \partial_x(\Delta v)-\partial_y((\Delta u)\cdot v_y)-u_{y}\cdot \partial_y(\Delta v),\notag\\
&=&\partial_x((\Delta u)\cdot v_x)-\partial_y (u_{y}\cdot (\Delta v))+(\Delta u)\cdot v_{xx}+u_x\cdot \partial_x(\Delta v)\notag\\
&&-u_{yy}\cdot (\Delta v)-\partial_y(\Delta u)\cdot v_y.\notag
\end{eqnarray}
Further, compute
\begin{eqnarray}
\partial_x((\Delta u)\cdot v_x)=-\partial_x(|\nabla u|^2u\cdot v_x),\notag
\end{eqnarray}
\begin{eqnarray}
-\partial_y((\Delta u)\cdot v_y)=\partial_y(|\nabla u|^2 u\cdot v_y),\notag
\end{eqnarray}
\begin{eqnarray}
&&(\Delta u)\cdot v_{xx}+u_x\cdot \partial_x(\Delta v)\notag\\
&=&-|\nabla u|^2 u\cdot v_{xx}+\partial_x[u_x\cdot (\Delta v)]-u_{xx}\cdot (\Delta v)\notag\\
&=&-|\nabla u|^2 u\cdot v_{xx}-\partial_x[u_x\cdot (|\nabla u|^2v+2(\nabla u\cdot \nabla v)u)]
+u_{xx}\cdot [|\nabla u|^2v+2(\nabla u\cdot \nabla v)u]\notag\\
&=&-|\nabla u|^2 u\cdot v_{xx}-\partial_x[|\nabla u|^2v \cdot u_x]+|\nabla u|^2v\cdot u_{xx}+2(\nabla u\cdot \nabla v)u\cdot u_{xx},\notag
\end{eqnarray}
and
\begin{eqnarray}
&&-u_{yy}\cdot (\Delta v)-\partial_y(\Delta u)\cdot v_y\notag\\
&=&u_{yy}\cdot (|\nabla u|^2v+2(\nabla u\cdot \nabla v)u)-\partial_y[(\Delta u)\cdot v_y]+(\Delta u)\cdot v_{yy}\notag\\
&=&|\nabla u|^2v\cdot u_{yy}+2(\nabla u\cdot \nabla v)u\cdot u_{yy}+\partial_y[|\nabla u|^2u\cdot v_y]-|\nabla u|^2u\cdot v_{yy}.\notag
\end{eqnarray}
Moreover, note that
\begin{eqnarray}
\partial_x(|\nabla u|^2u\cdot v_x)+\partial_x[|\nabla u|^2v \cdot u_x]=0,\notag
\end{eqnarray}
\begin{eqnarray}
\partial_y(|\nabla u|^2 u\cdot v_y)+\partial_y[|\nabla u|^2u\cdot v_y]=0,\notag
\end{eqnarray}
\begin{eqnarray}
-|\nabla u|^2 u\cdot (v_{xx}+v_{yy})=|\nabla u|^2 u\cdot [|\nabla u|^2v+2(\nabla u\cdot \nabla v)u]=2|\nabla u|^2 (\nabla u\cdot \nabla v),\notag
\end{eqnarray}
\begin{eqnarray}
|\nabla u|^2v\cdot [u_{xx}+u_{yy}]=-|\nabla u|^2v\cdot [|\nabla u|^2u]=0,\notag
\end{eqnarray}
\begin{eqnarray}
2(\nabla u\cdot \nabla v)u\cdot [u_{xx}+u_{yy}]=-2|\nabla u|^2(\nabla u\cdot \nabla v).\notag
\end{eqnarray}
Therefore, we have that
\begin{eqnarray}
\Delta (u_x\cdot v_x-u_y\cdot v_y)=0.\notag
\end{eqnarray}
Hence $u_x\cdot v_x-u_y\cdot v_y$ is a harmonic function. Since $v\in C^{1}(\mathbb C, T\mathbb S^2)$ and $\mathcal E(u)<\infty$, the claim \eqref{e:uvy} holds.

Similar computations yield
\begin{eqnarray}
\Delta (u_x\cdot v_y+u_y\cdot v_x)=0.\notag
\end{eqnarray}
Hence we have (\ref{e:uv-x}).

Summarizing \eqref{e:reduce} and the claim, we obtain that
\begin{eqnarray}
v_x-v\times u_y-u\times v_y=0.\notag
\end{eqnarray}
This completes the proof.
\end{proof}

We now compute the tangent map $D\mathcal S^{-1}$. Let $u=(u_1,u_2,u_3)\in\mathbb S^2$. Then
\begin{eqnarray}
\mathcal S^{-1}(u)=\left(\frac{u_1}{1-u_3},\frac{u_2}{1-u_3}\right)=\frac{u_1+iu_2}{1-u_3}.\notag
\end{eqnarray}
Let $u$ be a harmonic map with degree $m\ge 0$. Let $u(t)$ ($t\in [0,\epsilon)$) be a family of harmonic maps with $u'(0)=v$. Hence $v$ is a map from $\mathbb C$ to $T\mathbb S^2$ with $u\cdot v=0$.
Direct calculations yield that
\begin{eqnarray}
D\mathcal S^{-1}_{u}(v)=\left.\frac{d}{dt}\mathcal S^{-1}(u(t))\right|_{t=0}
=\left(\frac{u_1v_3-u_3v_1+v_1}{(1-u_3)^2}, \frac{u_2v_3-u_3v_2+v_2}{(1-u_3)^2}\right).\notag
\end{eqnarray}

On the other hand, assume that $u=\mathcal S(f)$ where $f(z)=f_1(z)+if_2(z)$ is a complex (irreducible) function, $g(z)=g_1(z)+ig_2(z)$. Then the tangent map of $\mathcal S$ is given by
\begin{eqnarray}\label{e:ds-holomorphic}
D\mathcal S_f(g)=\left(\begin{array}{c}
                2\dfrac{-f_1^{2}+f_2^{2}+1}{\left( f_1^{2}+f_2^{2}+1\right) ^{2}}g_1-\dfrac{4f_1f_2}{\left( f_1^{2}+f_2^{2}+1\right) ^{2}}g_2 \\
                 2\dfrac{f_1^{2}-f_2^{2}+1}{\left( f_1^{2}+f_2^{2}+1\right) ^{2}}g_2-\dfrac{4f_1f_2}{\left( f_1^{2}+f_2^{2}+1\right) ^{2}}g_1 \\
                \dfrac{4f_1}{\left( f_1^{2}+f_2^{2}+1\right) ^{2}}g_1+\dfrac{4f_2}{\left( f_1^{2}+f_2^{2}+1\right) ^{2}}g_2
              \end{array}\right).
\end{eqnarray}
We should point out that if $g$ is holomorphic with large orders of singularities, then $D\mathcal S_f(g)$ becomes unbounded. Whereas, if the order of $g$ is small but still larger than $f$, $D\mathcal S_f(g)$ may be bounded. Let us consider the following examples.

\begin{example}
1. Let $u(\cdot,a,b)=\mathcal S(\frac{1}{z-(a+ib)})$. Then
\begin{eqnarray}
u(z,a,b)=\left(
         \begin{array}{c}
           \dfrac{2(x-a)}{a^{2}-2ax+b^{2}-2by+x^{2}+y^{2}+1} \\
           -\dfrac{2(y-b)}{a^{2}-2ax+b^{2}-2by+x^{2}+y^{2}+1} \\
           -\dfrac{a^{2}-2ax+b^{2}-2by+x^{2}+y^{2}-1}{%
a^{2}-2ax+b^{2}-2by+x^{2}+y^{2}+1}\\
         \end{array}
       \right).\notag
\end{eqnarray}
Letting $a=b=0$, we find that
\begin{eqnarray}
v(z)=\left.\frac{d}{da}u(z;a,b)\right|_{a=0,b=0}=\left(
                     \begin{array}{c}
                       -2\dfrac{-x^{2}+y^{2}+1}{\left(
x^{2}+y^{2}+1\right) ^{2}} \\
                       \dfrac{-4xy}{\left(
x^{2}+y^{2}+1\right) ^{2}} \\
                       \dfrac{4x}{\left(x^{2}+y^{2}+1\right)
^{2}} \\
                     \end{array}
                   \right),\notag
\end{eqnarray}
The generate function of $\frac{d}{da}u(z;0,0)$ is
\begin{eqnarray}
D\mathcal S^{-1}_u(v(z))=\frac{x^{2}-y^{2}}{\left(
x^{2}+y^{2}\right) ^{2}}+i( \dfrac{-2x y}{\left(x^{2}+y^{2}\right) ^{2}})=\dfrac{1}{z^2}.\notag
\end{eqnarray}

2. Let $u(\cdot,a,b)=\mathcal S(\frac{1}{(z-(a+ib))^2})$. Direct calculations yield that
\begin{eqnarray}
v(z)=\left.\frac{d}{da}u(z,a)\right|_{a=0,b=0}=\left(
                            \begin{array}{c}
                              \dfrac{-4x(-x^{4}+2x^{2}y^{2}+3y^{4}+1)}{(x^{4}+2x^{2}y^{2}+y^{4}+1)^2} \\
                              \dfrac{-4y(3x^{4}+2x^{2}y^{2}-y^{4}-1)}{(x^{4}+2x^{2}y^{2}+y^{4}+1)^2} \\
                              \dfrac{8x(x^{2}+y^{2})}{(x^{4}+2x^{2}y^{2}+y^{4}+1)^2}\\
                            \end{array}
                          \right).\notag
\end{eqnarray}
Hence
\begin{eqnarray}
D\mathcal S^{-1}_u(v(z))=\dfrac{2x^{3}-6xy^{2}}{\left( x^{2}+y^{2}\right) ^{3}}+i\dfrac{2y^{3}-6x^{2}y}{\left( x^{2}+y^{2}\right) ^{3}}=\dfrac{2}{z^3}.\notag
\end{eqnarray}

3. Let $u=\mathcal S(f)$ where $f$ is a complex rational function. Compute
\begin{eqnarray}
D\mathcal S^{-1}_u(v)&=&\frac{u_1u_2+i(1-u_3-u_1^2)}{(1-u_3)^2}\notag\\
&=&f_1f_2+\frac{i}{2}(f_2^2-f_1^2+1)=-\frac{i}{2}f^2+\frac{i}{2}.\notag
\end{eqnarray}

\end{example}

We turn to compute the kernel of the linearized operators $L_{1,u}$.
\begin{prop}\label{p:cr}
Assume that $u=\mathcal S(f)$ is a harmonic map with degree $m\ge 0$, where a irreducible rational function is given by
$f(z)=\frac{q(z)}{p(z)}.$
Here $p,\,q$ are two polynomials with order $l,\,n$, respectively.
Then all bounded kernel maps of ${\rm ker}L_{1,u}$ are generated by changing of coefficients of $p$ and $q$.
In particular,
the real dimension of ${\rm ker}L_{1,u}$ is
$4m+2$.
\end{prop}
\begin{proof}
1. Assume ${\rm deg}(u)=m$. Then $m=l$ or $m=n$. We first consider $l=m$. Assume that
\begin{eqnarray}
p(z)=a_0z^m+a_1z^{m-1}+\cdots+a_{m-1}z+a_m,\quad q(z)=b_0z^n+b_1z^{n-1}+\cdots+b_{n-1}z+b_n,\notag
\end{eqnarray}
where $a_0,\cdots,a_m$ and $b_0,\cdots,b_n$ are complex numbers with $a_0\ne 0$ and $b_0\ne 0$. Let
\begin{eqnarray}
a_j=a_j^1+ia_j^2,\quad j=0,\cdots,m,\notag
\end{eqnarray}
\begin{eqnarray}
b_k=b_k^1+ib_k^2,\quad k=0,\cdots,n,\notag
\end{eqnarray}
where $a_j^1, a_j^2$ ($j=0,\cdots,m$) and $b_k^1,b_k^2$ ($k=0,\cdots,n$) are real numbers.
Set
\begin{eqnarray}
&&F(a_0,\cdots,a_m;b_0,\cdots,b_n;c_1,\cdots,c_{m-n})\notag\\
&=&\frac{c_1z^m+\cdots+c_{m-n}z^{n+1}+b_0z^n+b_1z^{n-1}+\cdots+b_{n-1}z+b_n}{a_0z^m+a_1z^{m-1}+\cdots+a_{m-1}z+a_m},\notag
\end{eqnarray}
where $c_1=c_1^1+ic_1^2,\cdots, c_{m-n}=c_{m-n}^1+ic_{m-n}^2$ are complex numbers.
Hence
\begin{eqnarray}\label{e:harmonic-family}
w(a_0,\cdots,a_m;b_0,\cdots,b_n;c_1,\cdots,c_{m-n}):=\mathcal S(F(a_0,\cdots,a_m;b_0,\cdots,b_n;c_1,\cdots,c_{m-n}))
\end{eqnarray}
is a family of harmonic maps with
$$\mathcal S(F(a_0,\cdots,a_m;b_0,\cdots,b_n;0,\cdots,0))=u.$$
Taking derivative on $w$ with respect to the real and imaginary part of the complex coefficients, from (\ref{e:ds-holomorphic}) we fine the following bounded kernels
\begin{eqnarray}
v^{1,c_{\nu}}(z)=D\mathcal S_u\left[\frac{z^{m-\nu+1}}{a_0z^m+a_1z^{m-1}+\cdots+a_{m-1}z+a_m}\right], \notag
\end{eqnarray}
\begin{eqnarray}
v^{2,c_{\nu}}(z)=D\mathcal S_u\left[i\frac{z^{m-\nu+1}}{a_0z^m+a_1z^{m-1}+\cdots+a_{m-1}z+a_m}\right],\notag
\end{eqnarray}
with $\nu=1,\cdots, m-n$, and
\begin{eqnarray}
v^{1,b_{k}}(z)=D\mathcal S_u\left[\frac{z^{n-k}}{a_0z^m+a_1z^{m-1}+\cdots+a_{m-1}z+a_m}\right],\notag
\end{eqnarray}
\begin{eqnarray}
v^{2,b_{k}}(z)=D\mathcal S_u\left[i\frac{z^{n-k}}{a_0z^m+a_1z^{m-1}+\cdots+a_{m-1}z+a_m}\right],\notag
\end{eqnarray}
with $k=0,\cdots,n$, and
\begin{eqnarray}
v^{1,a_{j}}(z)=D\mathcal S_u\left[-\frac{z^{m-j}}{(a_0z^m+a_1z^{m-1}+\cdots+a_{m-1}z+a_m)^2}\right],\notag
\end{eqnarray}
\begin{eqnarray}
v^{2,a_{j}}(z)=D\mathcal S_u\left[-i\frac{z^{m-j}}{(a_0z^m+a_1z^{m-1}+\cdots+a_{m-1}z+a_m)^2}\right],\notag
\end{eqnarray}
with $j=0,\cdots,m$. Note that $v^{1,b_{k}}$, $v^{2,b_{k}}$ and can be linearly represented by $v^{1,a_{0}},\cdots,v^{1,a_{m}},v^{2,a_{0}},\cdots,v^{2,a_{m}}$. Therefore, we obtain that
\begin{eqnarray}
{\rm span}\left\{\begin{array}{c}
v^{1,c_{1}},\cdots,v^{1,c_{m-n}};
v^{2,c_{1}},\cdots,v^{2,c_{m-n}} \\
v^{1,b_{0}},\cdots, v^{1,b_{n-1}};
v^{2,b_{0}},\cdots, v^{2,b_{n-1}} \\
v^{1,a_{0}},\cdots, v^{1,a_{m}};
v^{2,a_{0}},\cdots, v^{2,a_{m}}
\end{array}
 \right\}\subset{\rm ker}L_{1,u}.\notag
\end{eqnarray}

Similarly, we can compute the bounded kernels generated by the harmonic maps near $u$ of form (\ref{e:harmonic-family}) for $n=m$.

2. Next, we prove that for all $v\in {\rm ker}L_{1,u}$, $D\mathcal S^{-1}_u(v)$ satisfies Cauchy-Riemann equation in $\mathbb C\setminus\{z_1,\cdots,z_K\}$ where $z_1,\cdots,z_K$ are poles of $f$. That is, we want to verify that
\begin{eqnarray}
\left\{\begin{array}{c}
         \partial_x\left(\dfrac{u_1v_3-u_3v_1+v_1}{(1-u_3)^2}\right)=\partial_y\left(\dfrac{u_2v_3-u_3v_2+v_2}{(1-u_3)^2}\right) \\
         \partial_y\left(\dfrac{u_1v_3-u_3v_1+v_1}{(1-u_3)^2}\right)=-\partial_x\left(\dfrac{u_2v_3-u_3v_2+v_2}{(1-u_3)^2}\right)
       \end{array}
\right..\notag
\end{eqnarray}
That is,
\begin{eqnarray}
\left\{\begin{array}{l}
         \partial_x\left(\dfrac{(v\times u)_2+v_1}{(1-u_3)^2}\right)=\partial_y\left(\dfrac{(u\times v)_1+v_2}{(1-u_3)^2}\right) \\
         \partial_y\left(\dfrac{(v\times u)_2+v_1}{(1-u_3)^2}\right)=-\partial_x\left(\dfrac{(u\times v)_1+v_2}{(1-u_3)^2}\right)
       \end{array}
\right..\notag
\end{eqnarray}
So it is sufficient to check that
\begin{eqnarray}\label{e:cr2}
\left\{\begin{array}{l}
         \partial_x[(v\times u)_2+v_1](1-u_3)+2[(v\times u)_2+v_1]\partial_x u_3\\
         \quad\quad=\partial_y[(u\times v)_1+v_2](1-u_3)+2[(u\times v)_1+v_2]\partial_y u_3\\
        \partial_y[(v\times u)_2+v_1](1-u_3)+2[(v\times u)_2+v_1]\partial_y u_3\\
        \quad\quad=-\partial_x[(u\times v)_1+v_2](1-u_3)-2[(u\times v)_1+v_2]\partial_x u_3\\
       \end{array}
\right.
\end{eqnarray}

In fact, note that
\begin{eqnarray}\label{e:uxuy}
u_x=u\times u_y\quad \mbox{ and }\quad u_y=-u\times u_x.
\end{eqnarray}
Moreover, since $v\in {\rm ker}L_{1,u}$, we have that
\begin{eqnarray}
v_x-v\times u_y-u\times v_y=0,\quad
v_y+v\times u_x+u\times v_x=0.\notag
\end{eqnarray}
Using Lagrange identity and Jacobi identity for cross product, we obtain that
\begin{eqnarray}
\partial_x[(v\times u)_2+v_1]&=&(v_x\times u)_2+(v\times u_x)_2+(v_1)_x\notag\\
&=&2(v\times u_x)_2+(v_2)_y+(v_1)_x,\notag
\end{eqnarray}
\begin{eqnarray}
\partial_y[(u\times v)_1+v_2]&=&(u_y\times v)_1+(u\times v_y)_1+(v_2)_y\notag\\
&=&2(u_y\times v)_1+(v_1)_x+(v_2)_y.\notag
\end{eqnarray}
Hence
\begin{eqnarray}
&&\partial_x[(v\times u)_2+v_1](1-u_3)+2[(v\times u)_2+v_1]\partial_x u_3\notag\\
        && \quad\quad-\partial_y[(u\times v)_1+v_2](1-u_3)-2[(u\times v)_1+v_2]\partial_y u_3\notag\\
        &=&2[(v\times u_x)_2-(u_y\times v)_1](1-u_3)+2[(v\times u)_2+v_1]\partial_x u_3-2[(u\times v)_1+v_2]\partial_y u_3\notag.
\end{eqnarray}
Further,
\begin{eqnarray}
&&[(v\times u_x)_2-(u_y\times v)_1](1-u_3)\notag\\
&=&[(1-u_3)(-(u_3)_x)]v_1+[(1-u_3)(-(u_3)_y)]v_2+[((u_1)_x-(u_2)_y)(1-u_3)]v_3,\notag
\end{eqnarray}
and
\begin{eqnarray}
&&[(v\times u)_2+v_1]\partial_x u_3-[(u\times v)_1+v_2]\partial_y u_3\notag\\
&=&(1-u_3)(u_3)_x v_1+(1-u_3)(u_3)_y v_2 +(u_1 (u_3)_x-u_2 (u_3)_y)v_3.\notag
\end{eqnarray}
By (\ref{e:uxuy}), we find that
\begin{eqnarray}
&&((u_1)_x-(u_2)_y)(1-u_3)+u_1 (u_3)_x-u_2 (u_3)_y\notag\\
&=&(u_1)_x-(u_2)_y-(u_1)_xu_3+(u_2)_yu_3+u_1 (u_3)_x-u_2 (u_3)_y\notag\\
&=&(u_1)_x-(u_2)_y-(u\times u_y)_1-(u\times u_x)_2=0.\notag
\end{eqnarray}
Similar computations yield the second equation of (\ref{e:cr2}).

3. By using Equation (\ref{e:ds-holomorphic}) and Step 1 and 2, we have
all bounded kernels must be generated by harmonic maps near $u$.
This completes the proof.
\end{proof}

\begin{remark}
For the harmonic maps with negative degree, replacing holomorphic maps by anti-holomorphic ones, we can show the nondegeneracy similarly.
\end{remark}

\begin{proof}[Proof of Theorem \ref{t:general}]
Combining Proposition \ref{l:Lu-L1u} and Proposition \ref{p:cr}, we obtain the result in Theorem \ref{t:general}.
\end{proof}

As a consequence of Theorem \ref{t:general}, we can give an explicit formula for the bounded kernel of linearized operator.
\begin{proof}[Proof of Corollary \ref{t:main}]
Let $g_{k}(z)=z^{m-k}$, $k=0,1,\cdots,m$.
From (\ref{e:ds-holomorphic}),
\begin{eqnarray}
D\mathcal S_{U_m}(ig_k)=-2E_{k1},\quad D\mathcal S_{U_m}(g_k)=-2E_{k2}.\notag
\end{eqnarray}
They are generated by the families of form $u_t(z)=z^m+tz^{m-k}$, $t\in \mathbb C$.
Let $\tilde g_{\nu}(z)=z^{m+\nu}$ ($\nu=1,\cdots, m$). Then
\begin{eqnarray}
D\mathcal S_{U_m}(i\tilde g_\nu)=-2\tilde E_{\nu1},\quad D\mathcal S_{U_m}(\tilde g_\nu)=-2\tilde E_{\nu2}.\notag
\end{eqnarray}
These kernels are generated by $u_t(z)=\frac{z^m}{tz^{\nu}+1}$, $t\in \mathbb C$.
\end{proof}

\appendix

\section{}
In this appendix, we give some details of computations and an alternative proof of Corollary \ref{t:main} by ODE method.
\begin{lemma}\label{l:first-cr}
Map $u$ satisfies (\ref{e:JH}) if and only if
\begin{eqnarray}\label{e:cr}
\mathcal S^{-1}(u)=\dfrac{u_1+iu_2}{1-u_3}
\end{eqnarray}
satisfies Cauchy-Riemann equations in $\mathbb C\setminus\mathcal S^{-1}(N)$.
\end{lemma}
\begin{proof}
1. The equation (\ref{e:JH}) is equivalent to
\begin{eqnarray}
\left\{\begin{array}{c}
         \partial_{x}u_1 = u_2\partial_{y}u_3-u_3\partial_{y}u_2\\
         \partial_{x}u_2 = u_3\partial_{y}u_1-u_1\partial_{y}u_3 \\
         \partial_{x}u_3 = u_1\partial_{y}u_2-u_2\partial_{y}u_2
       \end{array}
\right.\quad \mbox{and}\quad u_{y}=-u\times u_{x}.\notag
\end{eqnarray}
Hence
\begin{eqnarray}
\left\{\begin{array}{c}
         \partial_{y}u_1 = u_3\partial_{x}u_2-u_2\partial_{x}u_3\\
         \partial_{y}u_2 = u_1\partial_{x}u_3-u_3\partial_{x}u_1 \\
         \partial_{y}u_3 = u_2\partial_{x}u_2-u_1\partial_{x}u_2
       \end{array}
\right..\notag
\end{eqnarray}
Therefore we find that
\begin{eqnarray}\label{e:Cauchy-Riemann}
\left\{\begin{array}{c}
                \dfrac{\partial_x u_1 + u_1\partial_x u_3-u_3\partial_x u_1}{(1-u_3)^2}=\dfrac{\partial_y u_2 + u_2\partial_y u_3-u_3\partial_{ x_2} u_2}{(1-u_3)^2} \\
                \dfrac{\partial_y u_1 + u_1\partial_y u_3-u_3\partial_y u_1}{(1-u_3)^2}=-\dfrac{\partial_x u_2 + u_2\partial_x u_3-u_3\partial_x u_2}{(1-u_3)^2}
              \end{array}.\right.
\end{eqnarray}
This yields that (\ref{e:cr}) satisfies Cauchy-Riemann equations in $\mathbb C\setminus\mathcal S^{-1}(N)$.

2. Conversely, Cauchy-Riemann equation yields (\ref{e:Cauchy-Riemann}).
Then after a stereographic projection,
\begin{eqnarray}
  D\mathcal S^{-1}_{u}(u_{x})= \left(
                                   \begin{array}{c}
                                     \dfrac{u_1\partial_{x}u_3-u_3\partial_{x}u_1+\partial_{x}u_1}{(1-u_3)^2} \\
                                     \dfrac{u_2\partial_{x}u_3-u_3\partial_{x}u_2+\partial_{x}u_2}{(1-u_3)^2} \\
                                   \end{array}
                                 \right),\notag
\end{eqnarray}
and
\begin{eqnarray}
&&D\mathcal S^{-1}_{u}(u\times u_{y})\notag\\
&=&\left(
\begin{array}{c}
\dfrac{u_1(u_1\partial_{y}u_2-u_2\partial_yu_1)-u_3(u_2\partial_yu_3-u_3\partial_yu_2)+u_2\partial_yu_3 - u_3\partial_yu_2}{(1-u_3)^2}\\
\dfrac{u_2(u_1\partial_yu_2-u_2\partial_yu_1)-u_3(u_3\partial_yu_1-u_1\partial_yu_3) +u_3\partial_yu_1-u_1\partial_yu_3}{(1-u_3)^2}  \\
                                   \end{array}
                                 \right) \notag\\
 &=& \left(
 \begin{array}{c}
 \dfrac{u_2\partial_yu_3-u_3\partial_yu_2+\partial_yu_2}{(1-u_3)^2} \\
 -\dfrac{u_1\partial_yu_3-u_3\partial_yu_1+\partial_yu_1}{(1-u_3)^2} \\
 \end{array}
 \right) = D\mathcal S^{-1}_{u}(u_{x}).\notag
\end{eqnarray}
Hence $u_x=u\times u_y$.
This completes the proof.
\end{proof}

\begin{remark}\label{r:anti-cr}
A similar computation yields that $u$ satisfies (\ref{e:anti-JH}) if and only if
\begin{eqnarray}\label{e:cr}
\mathcal S^{-1}(u)=\dfrac{u_1+iu_2}{1-u_3}
\end{eqnarray}
satisfies the anti-version of Cauchy-Riemann equations in $\mathbb C\setminus\mathcal S^{-1}(N)$.
\end{remark}

\begin{lemma}\label{l:equivalence}
$u$ satisfies (\ref{e:harmonic-map}) if and only if $u$ satisfies (\ref{e:JH}) or (\ref{e:anti-JH}).
\end{lemma}
\begin{proof}
1. ``$\Longrightarrow$'': Since $u_x=J^uu_y$, we find that
\begin{eqnarray}
u_{xx}=u_x\times u_y+u\times u_{xy}.\notag
\end{eqnarray}
Further note that $u_y=-J^u u_x$, then
\begin{eqnarray}
u_{yy}=-u_y\times u_x-u\times u_{xy}.\notag
\end{eqnarray}
It follows that
\begin{eqnarray}
\Delta u=2u_x\times u_y.\notag
\end{eqnarray}
Recall that it is the standard H-bubble equation.
On the other hand, from (\ref{e:JH}) we have that
\begin{eqnarray}
2(u_x\times u_y)\cdot u=-|\nabla u|^2.\notag
\end{eqnarray}
Therefore, (\ref{e:harmonic-map}) holds.

2. ``$\Longleftarrow$'': Compute
\begin{eqnarray}
\partial_{xx}u\cdot \partial_x u=\dfrac{1}{2}\partial_x |\partial_x u|^2,\notag
\end{eqnarray}
\begin{eqnarray}
\partial_{yy}u\cdot \partial_x u&=&\partial_x (\partial_{yy}u\cdot u)-(\partial_x\partial_{yy}u)\cdot u\notag\\
&=&\partial_x(\partial_y(\partial_yu\cdot u)-\partial_y u\cdot \partial_y u)-(\partial_y(\partial_{xy}u\cdot u)-\partial_{xy}u\cdot \partial_y u)\notag\\
&=&-\partial_x|\partial_y u|^2-\partial_y(\partial_{xy}u\cdot u)+\partial_{xy}u\cdot \partial_y u\notag\\
&=&-\partial_x|\partial_y u|^2+\partial_y(\partial_{x}u\cdot \partial_y u)+\partial_y(\partial_{x}u\cdot \partial_y u)-\partial_x u\cdot \partial_{yy}u\notag
\end{eqnarray}
It follows that
\begin{eqnarray}
\partial_{yy}u\cdot \partial_x u=-\dfrac{1}{2}\partial_x|\partial_y u|^2+\partial_y(\partial_{x}u\cdot \partial_y u).\notag
\end{eqnarray}
By (\ref{e:harmonic-map}), we have that
\begin{eqnarray}
0=\partial_{xx}u\cdot \partial_x u+\partial_{yy}u\cdot \partial_x u=\dfrac{1}{2}\partial_x( |\partial_x u|^2-|\partial_y u|^2)+\partial_y(\partial_{x}u\cdot \partial_y u).\notag
\end{eqnarray}
Similarly,
\begin{eqnarray}
0=\partial_{xx}u\cdot \partial_y u+\partial_{yy}u\cdot \partial_y u=\dfrac{1}{2}\partial_y( |\partial_y u|^2-|\partial_x u|^2)+\partial_x(\partial_{x}u\cdot \partial_y u).\notag
\end{eqnarray}
It follows that
\begin{eqnarray}
\Delta (\partial_{x}u\cdot \partial_y u)=0.\notag
\end{eqnarray}
From $|\partial_{x}u\cdot \partial_y u|\le \dfrac{1}{2}|\nabla u|^2$, mean value property and $\mathcal E(u)<\infty$, it holds that
\begin{eqnarray}
\partial_{x}u\cdot \partial_y u=0.\notag
\end{eqnarray}
Then recalling that $u_x\cdot u=0$ and $u_y\cdot u=0$, we may assume that
\begin{eqnarray}
ku_x=u\times u_y.\notag
\end{eqnarray}
That means
\begin{eqnarray}
ku_y=-u\times u_x.\notag
\end{eqnarray}
Note that
\begin{eqnarray}
ku_x\cdot u_x=(u\times u_y)\cdot u_x\quad \mbox{ and }\quad ku_y\cdot u_y=-(u\times u_x)\cdot u_y.\notag
\end{eqnarray}
Then we find that
\begin{eqnarray}
|u_x|=|u_y|.\notag
\end{eqnarray}
Therefore, $|k|=1$. This completes the proof.
\end{proof}

We now give a proof of Corollary \ref{t:main} by ODE.

Let
\[
u=\left(
\begin{array}
[c]{c}%
u_{1}\\
u_{2}\\
u_{3}%
\end{array}
\right)=\left(
\begin{array}
[c]{c}%
\cos\phi\sin Q\\
\sin\phi\sin Q\\
\cos Q
\end{array}
\right).
\]
where $\phi$ and $Q$ are two real functions on $\mathbb R^2$. In polar coordinates,
the energy functional becomes
\begin{eqnarray}\label{e:energy}
\mathcal E(u)=\mathcal E\left(\phi,Q\right)  =\int_{\mathbb R^2}\left(  \sin^{2}Q\left\vert \nabla\phi\right\vert
^{2}+\left\vert \nabla Q\right\vert ^{2}\right)  .
\end{eqnarray}
Hence if $\left(\phi,Q\right)$ is a critical point, then it satisfies the following system:
\begin{equation}\label{e:phi-q}
\left\{\begin{array}{l}
         \operatorname{div}\left( \sin^{2}Q\nabla\phi\right) =0, \\
         -2\Delta Q+\sin\left(  2Q\right)  \left\vert \nabla\phi\right\vert ^{2}=0.
       \end{array}\right.
\end{equation}
Choose $\phi=\phi_m=m\theta$, $Q=Q_m=\pi-2\arctan r^m$.
The linearized system of (\ref{e:phi-q}) at $(\phi_m,Q_m)$ is
\begin{equation}\label{e:phi-q2}
\left\{
\begin{array}
{l}%
\left(\dfrac{1}{2}\tan Q_m\right)\Delta\xi+(Q_m)_{r}\xi_{r}+\dfrac{m}{r^{2}}\eta_{\theta}=0,\\
-\Delta\eta+\dfrac{m^{2}}{r^{2}}\cos\left(  2Q_m\right)  \eta+\sin\left(
2Q_m\right)  \dfrac{m}{r^{2}}\xi_{\theta}=0.
\end{array}
\right.
\end{equation}
By some direct computations, (\ref{e:phi-q2}) becomes
\begin{equation}\label{e:ODE}
\left\{
\begin{array}
{l}%
- \xi_{rr}-\dfrac{1}{r}\xi_{r}-\dfrac{1}{r^{2}%
}\partial_{\theta}^{2}\xi -\dfrac{2m(1-r^{2m})}{r(1+r^{2m})}\xi_{r}+\dfrac{m(1-r^{2m})}{r^{m+2}%
}\eta_{\theta}=0,\\
-\eta_{rr}-\dfrac{\eta_{r}}{r}-\dfrac{1}{r^{2}}\partial_{\theta}^{2}\eta
+\dfrac{r^{4m}-6r^{2m}+1}{(1+r^{2m})^2}\dfrac{m^2}{r^{2}} \eta-\dfrac{4r^m(1-r^{2m})}{(1+r^{2m})^2}\dfrac{m}{r^{2}}\xi_{\theta}=0.
\end{array}
\right.
\end{equation}
In $\mathbb R^3$, the kernel maps are given by
\begin{eqnarray}\label{e:e-general-form}
E&=&\dfrac{d}{dt}\left.\left(
              \begin{array}{c}
                \cos (\phi_m+t\xi)\sin (Q_m+t\eta) \\
                \sin(\phi_m+t\xi) \sin (Q_m+t\eta)\\
                \cos (Q_m+t\eta) \\
              \end{array}
            \right)\right|_{t=0}\notag\\
            &=&\left(
                                   \begin{array}{c}
                                     -\xi\sin\phi_m\sin Q_m+\eta\cos \phi_m \cos Q_m \\
                                     \xi\cos\phi_m\sin Q_m+\eta\sin \phi_m \cos Q_m \\
                                     -\eta\sin Q_m \\
                                   \end{array}
                                 \right).
\end{eqnarray}

Assume that $\xi(r,\theta)=\xi_1(r)\cos(k\theta)+\xi_2(r)\sin(k\theta)$ and $\eta(r,\theta)=\eta_1(r)\cos(k\theta)+\eta_2(r)\sin(k\theta)$ with $k\in \mathbb N$. Then we have
\begin{equation}\label{e:general-m}
\left\{
\begin{array}
{l}%
- \xi_{1rr}-\dfrac{1}{r}\xi_{1r}+\dfrac{k^2}{r^{2}}\xi_{1}
-\dfrac{2m\left(  1-r^{2m}\right)  }{r\left(  1+r^{2m}\right)  }\xi_{1r}%
+\dfrac{mk\left(  1-r^{2m}\right)  }{r^{m+2}}\eta_{2}=0,\\
-\xi_{2rr}-\dfrac{1}{r}\xi_{2r}+\dfrac{k^2}{r^{2}}\xi_{2}
-\dfrac{2m\left(  1-r^{2m}\right)  }{r\left(  1+r^{2m}\right)  }\xi_{2r}%
-\dfrac{mk\left(  1-r^{2m}\right)  }{r^{m+2}}\eta_{1}=0,\\
-\eta_{1rr}-\dfrac{\eta_{1r}}{r}+\dfrac{k^2}{r^{2}}\eta_{1}+\dfrac{r^{4m}-6r^{2m}+1}{(1+r^{2m})^2}\dfrac{m^2}{r^{2}}\eta
_{1}+\dfrac{-4mk r^{m-2}\left(  1-r^{2m}\right)  }{\left(  1+r^{2m}\right)  ^{2}}\xi
_{2}=0,\\
-\eta_{2rr}-\dfrac{\eta_{2r}}{r}+\dfrac{k^2}{r^{2}}\eta_{2}+\dfrac{r^{4m}-6r^{2m}+1}{(1+r^{2m})^2}\dfrac{m^2}{r^{2}}\eta
_{2}-\dfrac{-4mk r^{m-2}\left(  1-r^{2m}\right)  }{\left(  1+r^{2m}\right)  ^{2}}\xi
_{1}=0.
\end{array}
\right.
\end{equation}
Therefore, the solutions of these systems are of form
\begin{equation*}
\left(\begin{array}{c}
        \xi \\
        \eta
      \end{array}
\right)=C_1\left(
          \begin{array}{c}
            \xi_1 \cos k\theta \\
            \eta_2 \sin k\theta \\
          \end{array}
        \right)+C_2\left(
          \begin{array}{c}
            \xi_2 \sin k\theta \\
            \eta_1 \cos k\theta \\
          \end{array}
        \right),
\end{equation*}
where $C_1,\,C_2$ are two arbitrary constants. By (\ref{e:e-general-form}), we have that
\begin{eqnarray}\label{e:E-general}
E&=&C_1\left[\xi_1\cos k\theta \sin Q_m E_1+\eta_2\sin k\theta E_2
\right]
+C_2\left[\xi_2\sin k\theta \sin Q_m E_1+\eta_1\cos k\theta E_2
\right]\\
&=& C_1\left[\xi_1\dfrac{2r^m}{1+r^{2m}}\cos k\theta  E_1+\eta_2\sin k\theta E_2
\right]
+C_2\left[\xi_2\dfrac{2r^m}{1+r^{2m}}\sin k\theta E_1+\eta_1\cos k\theta E_2
\right].\notag
\end{eqnarray}

System (\ref{e:general-m}) can be solved by different cases of $k$.
If $k=0$, then
\begin{equation*}
\xi_1(r)=C_1+C_2\dfrac{r^{4m}+4m\,r^{2m}\ln r-1}{2m\, r^{2m}},
\end{equation*}
\begin{equation*}
\eta_1(r)=C_3\dfrac{r^m}{1+r^{2m}}+C_4\dfrac{r^{4m}+4m\,r^{2m}\ln r-1}{r^m(1+r^{3m})}.
\end{equation*}
If $k\ne m$, then the solutions of (\ref{e:general-m}) can be given by
\begin{multline*}
\xi_1(r)=C_1\left(-\dfrac{1}{2r^k}\right)+C_2\dfrac{r^{k}}{2}\\
-C_3\dfrac{r^{k+4m}+\dfrac{4k+4m}{k}r^{k+2m}+\dfrac{k+m}{k-m}r^{k}}{2 r^{2m}}+C_4\dfrac{-r^{4m}+\dfrac{4 k-4m}{k}r^{2m}+\dfrac{k-m}{k+m}}{r^{k+2m}},
\end{multline*}
\begin{equation*}
\eta_2(r)=C_1\dfrac{1}{(1+r^{2m})r^{k-m}}+C_2\dfrac{r^{k+m}}{1+r^{2m}}
+C_3\dfrac{r^{k+4m}+\dfrac{k+m}{k-m}r^{k}}{(1+r^{2m})r^{m}}+C_4\dfrac{2r^{4m}+\dfrac{2k-2m}{k+m}}{(1+r^{2m})r^{k+m}},
\end{equation*}
and
\begin{multline*}
\xi_2(r)=C_5\left(-\dfrac{1}{2r^k}\right)+C_6\dfrac{r^{k}}{2}\\
-C_7\dfrac{r^{k+4m}+\dfrac{4k+4m}{k}r^{k+2m}+\dfrac{k+m}{k-m}r^{k}}{2 r^{2m}}+C_8\dfrac{-r^{4m}+\dfrac{4 k-4m}{k}r^{2m}+\dfrac{k-m}{k+m}}{r^{k+2m}},
\end{multline*}
\begin{equation*}
\eta_1(r)=-C_5\dfrac{1}{(1+r^{2m})r^{k-m}}-C_6\dfrac{r^{k+m}}{1+r^{2m}}
-C_7\dfrac{r^{k+4m}+\dfrac{k+m}{k-m}r^{k}}{(1+r^{2m})r^{m}}-C_8\dfrac{2r^{4m}+\dfrac{2k-2m}{k+m}}{(1+r^{2m})r^{k+m}}.
\end{equation*}
If $k=m$, the solutions are
\begin{eqnarray}
\xi_1(r)&=&C_1\dfrac{r^{2m}-1}{2r^{m}}-C_2\dfrac{1}{4m\, r^m}\notag\\
&&+C_3\dfrac{-r^{6m}+4m\, r^{4m}\ln r-7r^{4m}-4m\,r^{2m}\ln r-13r^{2m}-1}{4m\, r^{3m}}\notag\\
&&+C_4\dfrac{4m\,r^{4m}\ln r-7r^{2m}-1}{4m r^{3m}},\notag
\end{eqnarray}
\begin{multline*}
\eta_2(r)=C_1+C_2\dfrac{1}{2m(1+r^{2m})}+C_3\dfrac{r^{6m}+4m\,r^{4m}\ln r+r^{4m}+4m\,r^{2m}\ln r+5r^{2m}-1}{2m r^{2m}(1+r^{2m})}\\
+C_4\dfrac{4m \,r^{4m}\ln r-r^{2m}-1}{2m r^{2m}(1+r^{2m})},
\end{multline*}
and
\begin{eqnarray}
\xi_2(r)&=&C_5\dfrac{r^{2m}-1}{2r^{m}}-C_6\dfrac{1}{4m\, r^m}\notag\\
&&+C_7\dfrac{-r^{6m}+4m\, r^{4m}\ln r-7r^{4m}-4m\,r^{2m}\ln r-13r^{2m}-1}{4m\, r^{3m}}\notag\\
&&+C_8\dfrac{4m\,r^{4m}\ln r-7r^{2m}-1}{4m r^{3m}},\notag
\end{eqnarray}
\begin{multline*}
\eta_1(r)=-C_5-C_6\dfrac{1}{2m(1+r^{2m})}-C_7\dfrac{r^{6m}+4m\,r^{4m}\ln r+r^{4m}+4m\,r^{2m}\ln r+5r^{2m}-1}{2m r^{2m}(1+r^{2m})}\\
-C_8\dfrac{4m \,r^{4m}\ln r-r^{2m}-1}{2m r^{2m}(1+r^{2m})}.
\end{multline*}
Investigating these explicit solutions and using (\ref{e:E-general}),
we find that the bounded kernel maps are linearly combinations (\ref{e:ek1}) (\ref{e:ek3}) and (\ref{e:ek5}) .

\end{document}